\documentclass[]{interact}

\usepackage{epstopdf}
\usepackage[caption=false]{subfig}

\usepackage[longnamesfirst,sort]{natbib}
\bibpunct[, ]{(}{)}{;}{a}{,}{,}

\theoremstyle{plain}
\newtheorem{theorem}{Theorem}[section]

\newtheorem{proposition}[theorem]{Proposition}

\theoremstyle{definition}

\theoremstyle{remark}

\begin{document}
	\title{Stabilization of  Non-Diagonal
		Infinite-Dimensional Systems with Delay Boundary
		Control}
	
	\author{
		\name{Ionu\c t Munteanu\thanks{CONTACT I.M. Email: ionut.munteanu@uaic.ro}}
		\affil{Faculty of Mathematics, Alexandru Ioan Cuza University\\ and Octav Mayer Institute of Mathematics, Romanian Academy, \newline
			Bulevardul Carol I 11,  700506, Ia\c si, Romania}
	}
	
	\maketitle

	\begin{abstract}
		Here we deal with the stabilization problem of non-diagonal systems by boundary   control. In the studied setting, the boundary control input is subject to a constant delay.  We use the spectral decomposition method and split the system into two components: an unstable and a stable one. To stabilize the unstable part of the system, we connect, for the first time in the literature, the famous backstepping control design  technique  with the  direct-proportional control design. More precisely, we construct  a proportional open-loop stabilizer, then, by means of the Artstein transformation we close the loop.   At the end of the paper, an example is provided in order to illustrate the acquired results.
	\end{abstract}
	

	\section{Introduction}
	In this work, we study the stabilization problem  of the following  abstract boundary control system with delayed boundary control
	\begin{equation}\label{e1}\left\{\begin{array}{l}\frac{d}{dt}y(t)=\mathcal{A}y(t),\ t>0,\\
	\mathcal{B}y(t)=u(t-\tau),\ t>0,
	y(0)=y_o.\end{array}\right.\ \end{equation}
	A similar problem has been studied in the recent work \cite{krstic2}, under the additional assumption of  semi-simple eigenvalues. Here we drop this assumption and consider the general eigenvalues case.

	Throughout the paper we assume that $(\mathcal{H},\left<\cdot,\cdot\right>)$ and $(\mathcal{H}_0,\left<\cdot,\cdot\right>_0)$ are separable Hilbert spaces over the field $\mathbb{K}$, which is either $\mathbb{R}$ or $\mathbb{C}$. We denote by $\|\cdot\|$ the induced norm in $\mathcal{H}$.  In this work, we make the following assumptions:
	\begin{enumerate} 
		\item $\mathcal{A}:\mathcal{D}(\mathcal{A})\subset \mathcal{H}\rightarrow \mathcal{H}$ is a linear unbounded operator . \item $\mathcal{B}:\mathcal{D}(\mathcal{B})\subset \mathcal{H}\rightarrow \mathcal{H}_0,$ with $\mathcal{D}(\mathcal{A})\subset \mathcal{D}(\mathcal{B}),$ is a linear boundary operator. 
		\item $u:[-\tau,+\infty) \rightarrow \mathcal{H}_0$, with a known constant delay $\tau>0$ and $u|_{[-\tau,0)}=0,$ is the boundary control
	\end{enumerate}
	Besides this, we add the following assumptions:
	
	\noindent \textit{Assumption 1.} The operator $\mathcal{A}_0:=\mathcal{A}|_{\mathcal{D}(\mathcal{A})\cap \text{ker}(\mathcal{B})}$ is the generator of a $C_0-$ analytic semigroup, $\left\lbrace e^{t\mathcal{A}_0}:\ t\geq0\right\rbrace$ on $\mathcal{H}$. 
	Moreover, it has a countable set of eigenvalues $\left\lbrace \lambda_j\right\rbrace_{j=1}^\infty $ (repeated accordingly to their multiplicity), which are not necessarily semi-simple, i.e, the operator $\mathcal{A}_0$ may be non-diagonal. (The system $\left\lbrace \lambda_j\right\rbrace_{j=1}^\infty $
	is arranged decreasingly with respect to the real part.) Furthermore, from the corresponding eigenfunctions set, on may construct a Riesz basis in  $ \mathcal{H}$, $\left\lbrace \varphi_j \right\rbrace_{j=1}^\infty. $ 
	
	Let us recall that the semi-simple eigenvalue assumption (i.e., self-adjoint, or equivalently, diagonal operator $\mathcal{A}_0$)  is  related to the fact that the eigenfunction system forms a Riesz basis in $\mathcal{H}$. We notice that this assumption was a key point in the proofs in \cite{krstic2}. However, for non-self-adjoint operators, it may happen that the eigenfunctions  do not form a Riesz basis, but it is possible to construct  one. This is exactly the case of the example provided at the end of the paper.

	\noindent\textit{Assumption 2.} For each $\rho>0$, there exists $N\in\mathbb{N}^*$  such that: $\Re\lambda_j\leq -\rho$ for all $j\geq N+1$.

	Let us denote by $\left\lbrace \psi_j\right\rbrace_{j=1}^\infty, $ the associated bi-orthogonal system, in $\mathcal{H},$ of the basis $\left\lbrace\varphi_j \right\rbrace_{j=1}^\infty $. Namely, it holds $\left<\varphi_i,\psi_j\right>=\delta_{ij},\ i,j\in\mathbb{N}^*$, where $\delta_{ij}$ is the Kronecker symbol.
	
	\noindent\textit{Assumption 3.} We have $\left<\mathcal{A}_0\varphi_j,\psi_i\right>=0$ for all $j\geq N+1,\ 1\leq i\leq N.$
	
	Assumption 3 assures  that  the space spanned by the first $N$ functions $\left\{\varphi_i,\ 1\leq i\leq N\right\} $ is invariant under $\mathcal{A}_0$. 
	
	Next, we introduce the  matrix $\Lambda=(\lambda_{ij})_{i,j=1}^N$, where 
	$$\lambda_{ij}:=\overline{\left<\psi_i,\mathcal{A}_0\varphi_j\right>},\ 1\leq i,j\leq N.$$  $\Lambda$ is the Jordan matrix representing the operator $\mathcal{A}_0$ restricted to the space spanned by the first $N$ functions $\varphi_i$ .  The matrix $\Lambda$ may be non-diagonal. For $z\in \mathbb{C}$, we understand by $\overline{z}$ the complex conjugate of $z$.
	
	\noindent \textit{Assumption 4.} For $\gamma>0$  large enough, and for each $\beta\in R(\mathcal{B})$, there exists a unique solution, $D$, to the equation
	\begin{equation}\label{e2} -\mathcal{A}D+2\sum_{i,j=1}^N\lambda_{ij}\left<D,\psi_i\right>\varphi_j+\gamma D=0; \ \mathcal{B}D=\beta.\end{equation}
	
	This way, we may introduce the  operator  $D_\gamma: R(\mathcal{B})\rightarrow \mathcal{H}$, $D_\gamma\beta:=D$, $D$ solution to \eqref{e2}. Successively scalarly multiplying equation \eqref{e2} by $\psi_1,...,\psi_N$, we get
	\begin{equation}\label{e3}\left( \Lambda +\gamma I\right) \left(\begin{array}{c}\left<D_\gamma\beta,\psi_1\right>\\ \left<D_\gamma\beta,\psi_2\right>\\ ... \\ \left<D_\gamma\beta,\psi_N\right>\end{array} \right)= \left(\begin{array}{c}\left<\beta, l_1\right>_0\\ \left<\beta,l_2\right>_0\\ ...\\ \left<\beta,l_N\right>_0\end{array} \right), \end{equation}where $\left\lbrace l_1,l_2,...,l_N \right\rbrace $ are functions in $\mathcal{H}_0,$ which do not depend on $\gamma$ or $\beta$. Here, $I$ is the identity matrix of order $N$. We let
	$$L:=\left( \begin{array}{c}l_1, l_2, ..., l_N\end{array}\right)^T.$$ For a matrix $A$, we set $A^T$ for its transpose.

	\noindent \textit{Assumption 5.} The matrix $(\Lambda\ L)=\left[ L\  \Lambda L\  \Lambda ^2L\ ... \  \Lambda ^{N-1}L\right] $ has full rank in a nonzero measure set.

	Assumption 5 says, in fact,  that the couple $(\Lambda \ L)$  satisfies a Kalman-type controllability rank condition, which is not the classical Kalman condition associated to the system \eqref{e1}. In fact, it is weaker, as we shall see in the example below. In the diagonal case \cite{krstic2}, the classical Kalman controllability rank condition is assumed. This provides the existence of a so-called "feedback gain" matrix $K$, which is involved in the definition of the stabilizing controller.
	
	Let us consider the following example
	$$\begin{aligned}& y_t=\Delta y+c y,\ t>0, (x_1,x_2)\in (0,\pi)\times (0,\pi),\\&
	y(t,0,x_2)=u(t,x_2),\ t>0,\ x_2\in(0,\pi),\ y(t,x_1,x_2)=0 \text{ in rest}.\end{aligned}$$ It is well-know that the eigenvalues of the Dirichlet Laplace operator on the square $(0,\pi)\times (0,\pi)$ are $\left\{-(k^2+l^2),\  k,l\in\mathbb{N}\setminus\left\{0\right\}\right\},$ with the corresponding eigenfunctions $\left\{\varphi=\frac{2}{\pi}\sin(kx_1)\sin(lx_2),\ k,l\in\mathbb{N}\setminus\left\{0\right\}\right\}$ which form a Riesz basis in $H=L^2((0,\pi)^2).$ Let us project the equation on two eigenfunctions, namely on $$span\left\{\frac{2}{\pi}\sin(x_1)\sin(x_2),\ \frac{2}{\pi }\sin(x_1)\sin(2x_2)\right\}.$$ Also, consider the controller $u$ of the form
	$u(t,x_2)=v(t)\sin(x_2).$ In this case, the matrices $A_{N_0}$ and $B_{N_0}$ from \cite[Eq. (5)]{krstic2} are given by
	$$A_{N_0}=diag(-2+c,-5+c), \ B_{N_0}=\left(\begin{array}{c}\frac{2}{\pi}\int_0^\pi \sin^2 x_2dx_2\\ \frac{2}{\pi}\int_0^\pi \sin(2x_2)\sin x_2dx_2\end{array} \right) = \left(\begin{array}{c}1\\ 0\end{array} \right).$$ So, the classical Kalman controllability matrix has the form
	$$(A_{N_0}\ B_{N_0})=\left(\begin{array}{cc}1&(-2+c)\\
	0& 0\end{array} \right)$$ which, of course, does not have full rank. On the other hand, the new Kalman rank controllability condition,  Assumption 5, reads as: the matrix
	$$(\Lambda\ L)=\frac{2}{\pi}\left(\begin{array}{cc}\sin x_2 & (-2+c)\sin x_2\\ \sin (2x_2) &(-5+c)\sin(2x_2)\end{array} \right)$$ has full rank for $x_2\in J\subset (0,\pi)$, where $J$ is an nonempty interval. Easily seen, this holds true.   Of course, one may suggest  to look for a controller $u$ not depending on the function $\sin x_2$, but on another function such that the resulting classical Kalman matrix $(A_{N_0}\ B_{N_0})$ has full rank. In fact, this is exactly the main idea behind the proportional-type controllers: one  proves that there exists a functions set $\left\{\chi_k:\ k=1,2,...,N\right\}$ such that, once plugged the actuator
	$$u=\sum_{k=1}^Nv_k(t)\chi_k,$$ into the equations, it assures that the classical Kalman condition holds true. Clearly, this is not always possible, not even in the diagonal case. For the non-diagonal case, this task is even harder to realise. We refer to the recent work of Lasiecka and Triggiani \cite{la1}, where the authors study the problem of boundary stabilization  of the non-diagonal Navier-Stokes system, by proportional-type actuators (without delay). In order to ensure the existence of the set of functions $\left\{\chi_k:\ k=1,2,...,N\right\}$ which assure that the classical Kalman condition holds true, the authors are forced to plug into the equations an additional internal controller. So, for the stabilization of the non-diagonal Navier-Stokes equations, both internal and boundary actuators are needed, while the exact form of the functions set  $\left\{\chi_k:\ k=1,2,...,N\right\}$ and of the gain feedback matrix $K$ are not known. In the present paper, based on the new Kalman-type rank condition, Assumption 5, and on improvements on the control design in \cite{book}, we stabilize non-diagonal abstract systems only by boundary proportional-type controllers with delay, where the exact form of the functions set $\left\{\chi_k:\ k=1,2,...,N\right\}$  and the feedback gain matrix $K$ are given exactly. More precisely, $\chi_k=l_k,\ k=1,2,...,N,$ and $K=\sum_{k=1}^N(\Lambda^T+\gamma_kI)^{-1}A,$ see below for the notations. 
	
	At the end of the paper, in order to illustrate the results, we consider the problem of stabilization of the heat equation, with nonlocal boundary conditions, by boundary delayed  control. We show that all  the above assumptions hold true, especially we show that the new Kalman rank condition, Assumption 5, holds true. Of course, it would be interesting to study the  Navier-Stokes equations, but, to show that the new Kalman type condition holds true for this case is not a simple task at all. Thus, it is left for a subsequent work.

	Equations that take the form \eqref{e1} and obey Assumption 1-5  arise from physically meaningful problems, such as reaction-diffusion phenomena, phase turbulence phenomena and more.   For details, see \cite{krstic2} and the references therein.
	
	Here, our main objective is to design a feedback  law $u$ such that, once plugged into \eqref{e1}, it ensures the exponential stability of the corresponding closed-loop system.  Because we are only concerned in controlling the system from the starting time $t=0$, we assume that the system is uncontrolled for $t<0$. This is why it is imposed $u|_{[-\tau,0)}=0.$ Therefore, due to the delay $\tau$ in the control input of \eqref{e1}, the system remains open-loop for $t<\tau$ while the effect of the control input has an impact on the system only at times $t\geq \tau.$

	We should emphasize that, based on the ideas in the book of Krstic \cite{bookK},  Prieur and his co-workers, in the recent papers \cite{krstic1,krstic2},  provide substantial results regarding the boundary stabilization of parabolic-like equations by delayed controllers. They use the well-known backstepping technique and the diagonal assumption.   Here, we aim to deal with the non-diagonal case. In order to do this, our approach is to combine the backstepping technique in \cite{krstic2} with  the direct-proportional control design technique, developed by Barbu in  \cite{barbu}, and improved latter in \cite{book}. This is in fact the main novelty of the present work. Even if  the two methods, backstepping and direct-proportional, have lots of common  features, conceptually they are totally different. See the Appendix for more details about this.

	\section{The main result}
	
	In the spirit of proportional control design, the stabilizing feedback form will be given directly, following latter to prove the stability of the corresponding closed-loop system. The main result of this work is stated below.
	\begin{theorem}\label{t1} The unique solution of the closed-loop parabolic equation with boundary input delay
		\begin{equation}\left\{\begin{array}{l}\frac{d}{dt}y(t)=\mathcal{A}y(t),\ t>0,\\
		\mathcal{B}y(t)=-\displaystyle \sum_{k=1}^N\left<\left( \Lambda^T+{\gamma_k}I\right)^{-1} A\left[\sum_{j=0}^\infty (\mathcal{T}_\tau^jY)(t-\tau)\right], L\right>_{N},\\
		y(0)=y_o; \end{array} \right.\ \end{equation} is asymptotically exponentially converging to zero in $\mathcal{H}$.
		
		Here, $ A $  is introduced in \eqref{ie9} below, $Y$ is the vector consisting of the first $N$ modes of $y$, i.e.,
		$$Y(t):=\left(\left<y(t),\psi_1\right> ,\  \left<y(t),\psi_2\right>, ..., \  \left<y(t),\psi_N\right>\right)^T;$$ $0<\gamma_1<\gamma_2<...<\gamma_N$,  are $N$ large enough positive numbers; and   the operator $\mathcal{T}_\tau$ is defined in \eqref{po30} below. $\left<\cdot,\cdot\right>_N$ stands for the standard scalar product in $\mathbb{K}^N$.
	\end{theorem}

	The first  work dealing with input delayed unstable PDEs is \cite{5}, where a reaction-diffusion equation is considered, and a backstepping approach is developed to stabilize it.
	In \cite{4}, a stable PDE is controlled by means of a delayed bounded linear control operator (see also \cite{19} for a semilinear case). In the present work, the control operator is unbounded (being a boundary control) and the open-loop system is unstable. Unbounded control operators have been considered in \cite{16}, \cite{15}, \cite{14} for both wave and heat equations, where time varying delays are allowed with a bound on the time-derivative of the delay function. For the finite-dimensional case we refer to the results  by Mazenc and his co-workers \cite{maz,maz2} where they show the stability independent of how large the variation of the delay is. For other results on this subject we refer to \cite{2},\cite{3},\cite{400}. From a practical perspective, it is worth noting that input delays are generally uncertain and possibly time varying. The case
	of a distributed actuation scheme is even more complex
	since spatially-varying delays can arise due to network
	and transport effects that may vary among different spatial regions.   A first example of this situation occurs in
	the context of biological systems and population dynamics \cite{38}. An example of the latter can be found
	in the context of epidemic dynamics \cite{42}. The problem of stabilization of reaction-diffusion PDEs with internal time-and spatially-varying input delays is solved  in the recent work \cite{lahh}, via the backstepping technique.  For future, we intend to design a boundary control with time-and spatially-varying  delay for such equations.
	\section{Proof of the main result} As we already mentioned, in order to prove Theorem \ref{t1}, we shall  successfully merge the  two  techniques: the backstepping control design and the direct-proportional control design, respectively.  More exactly, we extend to the non-diagonal case the results in \cite{book}, then we  involve the method in \cite{bookK,krstic1,krstic2}.

	\noindent\textit{Proof of Theorem \ref{t1}.} Let us notice that, for $\gamma>0$ large enough, the matrix $\Lambda+\gamma I$ is nonsingular. Next, we fix $0<\gamma_1<\gamma_2<...<\gamma_N$, $N$ large enough positive numbers such that for each of them  equation \eqref{e2} is well-posed, and such that the matrices $\Lambda+\gamma_kI,\ k=1,2,...,N$ are all invertible.  We denote by $D_{\gamma_k},\ k=1,2,...,N,$ the corresponding  operators given by \eqref{e2}. 
	
	We  denote by $\textbf{B}$ the  Gram  matrix 
	\begin{equation}\label{ie5}\textbf{B}:=\left( \left<l_i,l_j\right>_0\right)_{1\leq i,j\leq N}. \end{equation} Next, we set
	\begin{equation}\label{bo}B_k:=\left( \Lambda+{\gamma_k}I\right)^{-1}\ \textbf{B}\ \overline{\left( \Lambda^T+{\gamma_k}I\right)^{-1}} ,\ k=1,...,N.\end{equation}
	(Note that, in the semi-simple case in \cite{book}, we defined $B_k$, see \cite[Eq. (2.21)]{book},  in terms of a diagonal matrix $\Lambda_{\gamma_k}.$ If $\Lambda$ is diagonal, then $ \left( \Lambda+{\gamma_k}I\right)^{-1}$ and $\Lambda_{\gamma_k}$ coincide.)
	
	The following result is essential in the definition of the feedback control. 
	\begin{proposition}\label{p1} The sum of $B_k$'s, i.e. $B_1+B_2+...+B_N$ is an invertible matrix.
	\end{proposition}
	\begin{proof}\ Let $ z= (z_1\ z_2\ ...\ z_N)^T\in \mathbb{K}^N$ such that $\displaystyle\left( \sum_{k=1}^NB_k\right) z=0.$ Hence,
		$$\left<B_1z,z\right>_N+\left<B_2z,z\right>_N+...+\left<B_Nz,z\right>_N=0.$$Or, equivalently by the definition of $B_k$
		\begin{equation}\label{lo1}\sum_{k=1}^N \left<\mathbf{B}\overline{\left( \Lambda^T+{\gamma_k}I\right)^{-1}}z, \overline{\left( \Lambda^T+{\gamma_k}I\right)^{-1}}z\right>_N=0.\end{equation}Recall that, $\mathbf{B}$, being a Gramian, is positive semi-definite. Thus,
		$$\left<\mathbf{B}\overline{\left( \Lambda^T+{\gamma_k}I\right)^{-1}}z, \overline{\left( \Lambda^T+{\gamma_k}I\right)^{-1}}z\right>_N=0,\ \forall k=1,2,...,N.$$
		Due to the definition of $\mathbf{B}$, the above implies that
		$$\left<\left<L,\overline{(\Lambda^T+\gamma_kI)^{-1}}z\right>_N,\left<L,\overline{(\Lambda^T+\gamma_kI)^{-1}}z\right>_N\right>_0=0,\ k=1,2,...,N.$$ Which yields
		$$\left<L,\overline{\left( \Lambda^T+{\gamma_k}I\right)^{-1}}z\right>_N=0 \text{ almost everywhere},\ k=1,2,...,N.$$ Or, equivalently,
		$$\left<\left(\Lambda+\gamma_kI\right) ^{-1}L,z\right>_N=0 \text{ almost everywhere},\ k=1,2,...,N. $$Recall that $L:=\left( \begin{array}{c}l_1\  l_2\  ...\ l_N\end{array}\right)^T.$ The above can be viewed as a linear system in $U$, of order $N$, with the unknowns $z_1,z_2,...,z_N$, which are constants in $U$. It has only the trivial solution if and only if the determinant of the matrix of the system
		$$\det\left[ \left( \Lambda+\gamma_1I\right) ^{-1}L\ \  \left( \Lambda+\gamma_2I\right) ^{-1}L\  \ ... \ \left(\Lambda+\gamma_NI\right) ^{-1}L \right]\neq 0,$$in a nonzero measure set. Performing elementary transformations in the determinant, namely subtracting from each column $k=2,3,...,N$ the first column, i.e.
		$$\left( \Lambda+\gamma_kI\right) ^{-1}L-\left(  \Lambda+\gamma_{1}I\right) ^{-1}L=(\gamma_1-\gamma_k)\left(  \Lambda+\gamma_1I\right) ^{-1}\left(  \Lambda+\gamma_kI\right) ^{-1}L,$$ $k=2,...,N,$ the above is equivalent with
		$$\det\left[ \left(  \Lambda+\gamma_1I\right) ^{-1}L\ \ \left(  \Lambda+\gamma_1I\right) ^{-1}\left(  \Lambda+\gamma_2I\right) ^{-1}L\ ...\ \left(  \Lambda+\gamma_1I\right) ^{-1}\left(  \Lambda+\gamma_NI\right) ^{-1}L\right]$$is not equal zero. This holds true if and only if
		$$\det\left[ L\ \  \left(  \Lambda+\gamma_2I\right) ^{-1}L\ \ \left(  \Lambda+\gamma_3I\right) ^{-1}L ...\ \left(  \Lambda+\gamma_NI\right) ^{-1}L\right]\neq 0.$$Similar actions as above, namely subtracting from each column $k=3,4,...,N$ the second column, then multiplying the result by $ \Lambda+\gamma_2I $, lead to the equivalent condition
		$$\det\left[ \left(  \Lambda+\gamma_2I\right) L\ \ L \ \ \left(  \Lambda+\gamma_3I\right) ^{-1}L\  ...\ \left(  \Lambda+\gamma_NI\right) ^{-1}L\right]\neq0. $$We go on like this and arrive at
		$$\det\left[\left(  \Lambda+\gamma_NI\right)\cdot ...\cdot \left(  \Lambda+\gamma_2I\right) L\ \ \left(  \Lambda+\gamma_NI\right)\cdot ...\cdot \left( \Lambda+\gamma_3I\right)L\ \ ...\ L   \right]\neq0. $$Again elementary transformations: subtracting from the $(N-1)$th column the $N$th column multiplied by $\gamma_N$, yield
		$$\det\left[\left( \Lambda+\gamma_NI\right)\cdot ...\cdot \left(  \Lambda+\gamma_2I\right) L\ \ ...\ \   \left(  \Lambda+\gamma_NI\right)\left(  \Lambda+\gamma_{N-1}I\right) L\ \   \Lambda L\ \ L\right] $$
		Then, subtracting from the $(N-2)$th column the $(N-1)$th column multiplied by $\gamma_N+\gamma_{N-1}$ and the $N$th column multiplied by $\gamma_N\gamma_{N-1}$ we obtain that
		$$\det\left[ \left(  \Lambda+\gamma_NI\right)\cdot ...\cdot \left(  \Lambda+\gamma_2I\right) L\ \ ...\  \Lambda^2L \ \   \Lambda L \ \  L\right]\neq 0. $$The procedure goes in a similar way until we get that the above is equivalent with the fact that the determinant
		$$\det\left[  \Lambda^{N-1}L\ \  \Lambda^{N-2}L\ \ ...\ \  \Lambda L\ \ L\right]\neq0. $$ In virtue of Assumption 5, this holds true. We conclude that,  $\displaystyle\left( \sum_{k=1}^NB_k\right) z=0$ if and only if $z=0$. Or, in other words, the matrix $\displaystyle \sum_{k=1}^NB_k$ is invertible. 
	\end{proof}

	So, we may well-define the matrix
	\begin{equation}\label{ie9}A:=(B_1+B_2+...+B_{N})^{-1}.\end{equation} 
	Next, for $U:[-\tau,\infty)\rightarrow \mathbb{K}^N$, we set
	\begin{equation}\label{ie10}\left[u_k(U)\right](t-\tau):=-\left<\left(\Lambda^T+\gamma_kI\right)^{-1}AU(t-\tau), L\right>_{N},\ t\geq-\tau, \end{equation}$k=1,2,...,N$. Then,  introduce $u$ as 
	\begin{equation}\label{mio27}\begin{aligned}\left[ u(U) \right](t-\tau):&=\left[ u_1(U)+u_2(U)+...+u_{N}(U)\right](t-\tau) \\&
	=-\sum_{k=1}^N\left<\left( \Lambda^T+\gamma_kI\right) ^{-1}AU(t-\tau), \ L\right>_{N}.
	\end{aligned}\end{equation}   The vector $U(t)$ will be constructed below via the Artstein transform \cite{100}, and, at the end of the day will be a function of 
	$Y$ (this way we close the loop). 
	
	For latter purpose, let us show that
	\begin{equation}\label{e27}\left(\begin{array}{c}\left<D_{\gamma_k}u_k,\psi_1\right>\smallskip\\ \left<D_{\gamma_k}u_k,\psi_2\right>\\..................\\\left<D_{\gamma_k}u_k,\psi_N\right>\end{array}\right)=-B_k\textbf{A}\left(\begin{array}{c}\left<U(t-\tau),\psi_1\right>\smallskip\\
	\left<U(t-\tau),\psi_2\right>\\................\\\left<U(t-\tau),\psi_N\right>\end{array}\right),\end{equation} for all $k=1,...,N$. This is indeed so. We have by \eqref{e3} that
	$$\left( \begin{array}{c}\left<D_{\gamma_k}u_k,\psi_1\right>\\ \left<D_{\gamma_k}u_k,\psi_2\right>\\ ...\\ \left<D_{\gamma_k}u_k,\psi_N\right>\end{array}\right) =(\Lambda+\gamma_kI)^{-1}
	\left(\begin{array}{c}\left<u_k, l_1\right>_0\\ \left<u_k,l_2\right>_0\\ ...\\ \left<u_k,l_N\right>_0.\end{array} \right)$$
	Which, by taking into account \eqref{ie5} and \eqref{ie10}, yields
	$$\left( \begin{array}{c}\left<D_{\gamma_k}u_k,\psi_1\right>\\ \left<D_{\gamma_k}u_k,\psi_2\right>\\ ...\\ \left<D_{\gamma_k}u_k,\psi_N\right>\end{array}\right) =-(\Lambda+\gamma_kI)^{-1}\mathbf{B}\overline{\left(\Lambda^T+\gamma_kI\right)^{-1}} A\left(\begin{array}{c}\left<U(t-\tau),\psi_1\right>\smallskip\\
	\left<U(t-\tau),\psi_2\right>\\........\\\left<U(t-\tau),\psi_N\right>\end{array}\right) ,$$and so, by \eqref{bo},  \eqref{e27} is proved.

	Next, we plug this feedback into equation \eqref{e1}, and argue similarly as in \cite[Eqs. (2.27)-(2.29)]{book}. This way, we equivalently rewrite \eqref{e1} as an internal-type control problem:
	\begin{equation}\label{v1}\begin{aligned} \frac{d}{dt}y(t)=&\mathcal{A}y(t)+\sum_{k=1}^N(\mathcal{A}+\gamma_kI)D_{\gamma_k}u_k(U(t-\tau)) 
	\\& -2\sum_{i,j,k=1}^N\lambda_{ij} \left<D_{\gamma_k}u_k(U(t-\tau)),\psi_i\right>\varphi_j,\ t>0.
	\end{aligned}\end{equation}
	We apply the well-know  projection method to \eqref{v1}, and split it into two systems: one of which is unstable but is finite-dimensional, and the other one which is infinite-dimensional but  is stable. We will take care  of the finite-dimensional unstable part, only. 
	
	Projecting equation \eqref{v1} on the space spanned by $\left\{\psi_i\right\}_{i=1}^N,$ and taking into account the relation \eqref{e27}, we arrive at
	\begin{equation}\label{op2}\frac{d}{dt}Y(t)=\Lambda Y(t) -\left[\Lambda +\sum_{k=1}^N\gamma_kB_kA\right]U(t-\tau),\ t>0.\end{equation}We denote by 
	$$C:=-\Lambda -\sum_{k=1}^N\gamma_kB_kA.$$ Thus, \eqref{op2} becomes
	\begin{equation}\label{op3}\frac{d}{dt}Y(t)=\Lambda Y(t)+CU(t-\tau),\ t>0.\end{equation}This is an equation of the same type as \cite[Eq. (2.1)]{bookK}. Hence, we may apply the backstepping design  via a trasport PDE technique described in \cite{bookK}. But, before this, let us notice that  if there is no delay ($\tau=0$) in \eqref{op3}, we may take $U\equiv Y$ to obtain an exponentially stable system. Indeed, in this case, \eqref{op3} reads as
	\begin{equation}\label{op4} \frac{d}{dt}Y(t)=\Lambda Y(t)+CY(t)=-\sum_{k=1}^N\gamma_kB_kAY(t),\ t>0.\end{equation} We have
	$$\Lambda +C =-\sum_{k=1}^N\gamma_kB_kA=-\gamma_1I+\sum_{k=2}^N(\gamma_1-\gamma_k)B_kA,$$ by virtue of the fact that $A=(B_1+...+B_N)^{-1}.$ Then, for any $z\in \mathbb{K}^N$, we have
	$$\left<(\Lambda+C)z,Az\right>_N=-\gamma_1\|A^\frac{1}{2}z\|_N^2+\sum_{k=2}^N(\gamma_1-\gamma_k)\left<B_kAz,Az\right>_N.$$ (Here, $A^\frac{1}{2}$ is the square root matrix of $A$, which can be defined because $A$ is symmetric and positive definite.) Recalling the definition of $B_k$, see also relation \eqref{lo1}, we see that 
	$$(\gamma_1-\gamma_k)\left<B_kAz,Az\right>_N\leq 0,\ k=2,3,...,N.$$Consequently,
	$$\left<(\Lambda+C)z,Az\right>_N\leq-\gamma_1\|A^\frac{1}{2}z\|_N^2,\ \forall z\in \mathbb{K}^N.$$ Using this, and taking into account that $A$ is symmetric and positive definite, we get after scalarly multiplying equation \eqref{op4} by $AY$ that
	$$\|Y(t)\|^2_N\leq Ce^{-\gamma_1 t}\\Y(0)\|^2,\ \forall t\geq0.$$   We shall see below that this fact will eliminate the need of the matrix $K$ from the pole shifting theorem, used in \cite{bookK}. 
	
	Following the ideas in \cite[Section 2.2]{bookK}, we model the delay in \eqref{op3} by the following first-order hyperbolic PDE
	$$\partial_tZ(s,t)=\partial_sZ(s,t);\ Z(\tau,t)=U(t).$$ The system \eqref{op3} can now be written as
	$$\frac{d}{dt}Y(t)=\Lambda Y(t)+CZ(0,t).$$ Then, we consider the backstepping transfomation
	$$W(s,t)=Z(s,t)-\int_0^sQ(s,r)Z(r,t)dr-\Gamma(s)Y(t)$$ with which we want to map the above system into the target system
	$$\left\{\begin{array}{l}\frac{d}{dt}Y(t)=(\Lambda+C) Y(t)+CW(0,t),\\
	\partial_tW(s,t)=\partial_sW(s,t),\\
	W(\tau,t)=0.\end{array}\right.\ $$
	Performing similar computations as in \cite[Eqs. (2.26)-(2.39)]{bookK}, and recalling  that $\Lambda+C$ is Hurwitz, we get that
	$$Q(s,t)=e^{(s-t)\Lambda} \text{ and }\Gamma(s)=e^{s\Lambda}C.$$ So, the stabilizing control is given in an implicit form as
	$$U(t)=Y(t)+\int_{\max(t-\tau,\tau)}^te^{(t-\tau-s)\Lambda}CU(s)ds,\ t\geq0,$$ and $U(t)=0,\ t\in[-\tau,0).$ We have to solve the above fixed point implicit equation. To this end,  for any integrable vector $F$ on $\mathbb{R}$, we define
	\begin{equation}\label{po30}(\mathcal{T}_\tau F)(t):=\int_{\max(t-\tau,\tau)}^te^{(t-\tau-s)\Lambda}CF(s)ds.\end{equation} It follows that $U(t)$ can be written as the Neumann series
	$$U(t)=\sum_{j=0}^\infty (\mathcal{T}_\tau^jY)(t).$$ We can show the convergence of this series in a similar manner as in \cite[Lemma 3]{krstic1}. 
	
	Finally, arguing as in \cite[Section IV]{krstic2}, we conclude that, once we plug the feedback
	$$u(t-\tau)=-\displaystyle \sum_{k=1}^N\left<\left( \Lambda^T+\gamma_kI\right) ^{-1}A\left[\sum_{j=0}^\infty (\mathcal{T}_\tau^jY)(t-\tau)\right], \ L\right>_{N},$$ into equation \eqref{e1} it yields the desired result of the Theorem \ref{t1}. The details are omitted. 
	
	\section{Stabilization of the heat equation with nonlocal boundary conditions and delay control}
	As an application, let us consider the following nonlocal boundary value problem
	\begin{equation}\label{o1} \left\lbrace \begin{array}{l}y_t(t,x)-y''(t,x)-cy(t,x)=0,\ t>0, x\in(0,\pi),\\
	
	y(t,0)=u(t-\tau),\\
	y'(t,0)+y'(t,\pi)+\alpha y(t,\pi)=0,\ t>0,\\
	
	y(0,x)=y_o(x),\ x\in(0,\pi).
	\end{array} \right. \end{equation} Here, $'$ stands for the spatial  derivative, i.e., $f'=\frac{\partial f}{\partial_x}.$ $\alpha,c$ are some positive numbers. Boundary-value problems, with  two, three, or multi-point nonlocal boundary conditions, arise naturally in thermal conduction, semiconductor or hydrodynamic problems. For details see \cite{222}
	and the references therein. As far as we know, concerning the boundary stabilization of PDEs with non-local boundary conditions, there exist only the result in \cite{ionnonloc}, while for the case with delay in the control there is no result in the literature.

	In this case $\mathcal{H}=L^2(0,\pi)$ and $\mathcal{H}_0=\mathbb{R}$. The operator $\mathcal{A}:\mathcal{D}(\mathcal{A})\subset L^2(0,\pi)\rightarrow L^2(0,\pi)$ is given as
	$$\mathcal{A}y:=y''+cy,\ \forall y\in\mathcal{D}(\mathcal{A}),$$where $\mathcal{D}(\mathcal{A})$ is the set
	$$\left\lbrace y\in H^2(0,\pi):\ y(0)=0,\ y'(0)+y'(\pi)+\alpha y(\pi)=0\right\rbrace.$$
	By \cite{1}, we know that $\mathcal{A}$ has a countable set of eigenvalues $\left\lbrace \lambda_j\right\rbrace_{j=0}^\infty $ described as follows
	$$\lambda_j=\left\lbrace \begin{array}{ll}-(2k+1)^2+c & ,\text{if } j=2k, \ k\in\mathbb{N},\\ -(2\beta_k)^2+c & ,\text{if } j=2k+1,\ k\in\mathbb{N} .\end{array} \right. $$ Here, $\beta_k,\ k\in\mathbb{N},$ are the roots of the equation 
	$$\cot(\beta\pi)=-\frac{\alpha}{2\beta}.$$ Easily seen, given $\rho>0$, there exists $N\in\mathbb{N}$ such that 
	$$-\rho>\lambda_{2N+2}>\lambda_{2N+3}>....$$The corresponding eigenfunctions are precisely given in \cite{1}. More precisely they are $\left\lbrace w_{k1},\ w_{k2}\right\rbrace_{k=0}^\infty$, where
	$$w_{k1}=\sin((2k+1)x),\  \text{and }w_{k2}=\sin(2\beta_kx), $$ $k\in\mathbb{N}.$ As stated and proved in \cite[Lemma 4.1]{1}, it happens that the above system does not form a basis in $L^2(0,\pi)$.  That is why, in \cite{1}, the authors introduced the following new set of functions
	\begin{equation}\label{e3?}\varphi_j(x)=\left\lbrace \begin{array}{ll} w_{k1}(x)&,\text{if }j=2k,\\
	\left[ w_{k2}(x)-w_{k1}(x)\right](2\delta_k)^{-1}&,\text{if }j=2k+1,\end{array} \right.  \end{equation}$ k\in\mathbb{N}.$ Here, $\delta_k:=\beta_k-k-\frac{1}{2}.$ Then, in \cite[Lemma 5.1]{1} they proved that the system $\left\lbrace \varphi_j\right\rbrace _{j=0}^\infty$ forms a Riesz basis in $L^2(0,\pi)$. Moreover, they do also precised the bi-orthonormal system to $\left\lbrace \varphi_j\right\rbrace _{j=0}^\infty$, which is given by
	\begin{equation}\label{e4}\psi_j(x)=\left\lbrace \begin{array}{ll}v_{k2}(x)+v_{k1}(x)&,\text{if }j=2k,\\
	2\delta_kv_{k2}(x)&,\text{if }j=2k+1,\end{array}  \ \ \ \ \ k\in\mathbb{N}. \right. \end{equation} Where
	$$v_{k1}(x)=\frac{2}{\pi}\left\lbrace \sin((2k+1)x)-\frac{2k+1}{\alpha}\cos((2k+1)x)\right\rbrace ,$$and
	$$v_{k2}=C_{k2}\left\lbrace \sin(2\beta_kx)-\frac{2\beta_k}{\alpha}\cos(2\beta_kx)\right\rbrace,\ k=0,1,2,...$$Here, $C_{2k}$ is some constant which assures that the systems $ \left\lbrace w_{k1},w_{k2}\right\rbrace_{k=0}^\infty $ and $\left\lbrace v_{k1},v_{k2}\right\rbrace_{k=0}^\infty$  are bi-orthonormal.
	
	In the present case, the matrix $\Lambda$ is given in \cite[Eq. (2.20)]{ionnonloc}, as
	$$
	\left(\begin{array}{ccccccc}\lambda_0&2\beta_0+1&0&\dots&0&0&0\\
	0&\lambda_1&0&\dots&0&0&0\\
	0&0&\lambda_2&2\beta_1+3&0&\dots&0\\
	\ddots&\ddots&\ddots&\ddots&\ddots&\ddots&\ddots\\
	0&0&0&\dots&\lambda_{2N-1}&0&0\\
	0&0&0&0&\dots&\lambda_{2N}&2\beta_N+2N+1\\
	0&0&0&\dots&0&0&\lambda_{2N+1} \end{array}\right)$$ 
	The lifting operator is defined as: for $\gamma>0$, let $D$ be the solution to the equation 
	\begin{equation}\label{oe5}\left\lbrace \begin{array}{l}\begin{aligned}&-D''(x)-cD(x)-2\sum_{j=0}^{2N+1}\lambda_j\left<D,\psi_j\right>\varphi_j\\&
	-2\sum_{j=0}^{N}(2\beta_j+2j+1)\left<D,\psi_{2j+1}\right>\varphi_{2j}+\gamma D=0,\ x\in(0,\pi),\end{aligned}\\
	\\
	D(0)=1,\ D'(0)+D'(\pi)+\alpha D(\pi)=0.\end{array} \right. \end{equation}   In \cite[Lemma 2.1]{ionnonloc}, it is shown the well-posedness of this equation. 
	
	So far, Assumptions 1-4 are verified.  Let us show that Assumption 5 holds true as-well.  By \cite[Eq. (2.9)]{ionnonloc}, we know that
	\begin{equation}\label{e15}\left\lbrace \begin{array}{l}\left<D,\psi_{2j}\right>= \frac{1}{\gamma-\lambda_{2j}}\left[\psi'_{2j}(0)+\frac{2\beta_j+2j+1}{\gamma-\lambda_{2j+1}}\psi'_{2j+1}(0) \right],\\
	\\
	\left<D,\psi_{2j+1}\right>=\frac{1}{\gamma-\lambda_{2j+1}} \psi'_{2j+1}(0),\end{array}\right.  \end{equation}for $j=0,1,2,...,N$.  Or, equivalently,
	$$\left(\begin{array}{c}\left<D,\psi_0\right>\\ \left<D,\psi_1\right>\\ ... \\ \left<D,\psi_{2N+1}\right>\end{array} \right)=\left(\Lambda+\gamma I\right)^{-1}\left( \begin{array}{c} \psi_0'(0)\\ \psi_1'(0)\\... \\ \psi_{2N+1}'(0)\end{array}\right)   $$ Hence, in this case $l_j=\psi_j'(0),\ j=0,1,...,2N+1$.  Since $l_j\neq0$ for all $j=0,1,...,2N+1$, by the special form of $\Lambda$, it is easy to see that the couple $(\Lambda \ L)$ satisfies the Kalman rank condition. Therefore, Assumption 5 is fullfiled. Consequently, the following result holds true

	\begin{theorem}
		The unique solution of the closed-loop parabolic equation with nonlocal boundary values and input delay
		\begin{equation}\left\{\begin{array}{l}\partial_ty(t,x)= y''(t,x)+cy(t,x),\ t>0, x\in (0,\pi),\\
		y(t,0)=-\displaystyle \sum_{k=1}^N\left<\left( \Lambda^T+\gamma_k I\right) ^{-1}A\left[\sum_{j=0}^\infty (\mathcal{T}_\tau^jY)(t-\tau)\right],\  L\right>_{2N+2},\\y'(t,0)+y'(t,\pi)+\alpha y(t,\pi)=0,\ t>0,\\
		\\
		y(0,x)=y_o(x),\ x\in(0,\pi).
		\end{array} \right. \end{equation} is asymptotically exponentially converging to zero in $L^2(0,\pi).$ 
	\end{theorem}
	\section{Conclusions} Merging the two techniques: backstepping design and direct-proportional design, we  solved the stabilization problem for non-diagonal systems with boundary delay control.    
	
	Time-delays are a delicate issue in engineering systems, which often involve either communications lags or physical dead-time which reveals troublesome in the design and tuning of feedback control laws. Therefore, a robust controller is needed. On this subject, in the finite-dimensional case, there are many important results obtained by Bresch-Pietri and co-workers, see e.g. \cite{bresch1,bresch2}.  Since the method we applied here consists of the split of the infinite-dimensional system into an unstable finite-dimensional one and an infinite-dimensional stable one, we can try to apply the complex robust control design methods in the afore-mentioned papers, to the finite-dimensional part, in order to construct a robust control with delay. But then, when returning to the initial infinite-dimensional system, the main problem would be to show that the robust controller assures its stability as-well. This is not a simple task and  is left for a subsequent work. 
	
	Numerical examples to show the effectiveness of the proportional controller were performed in \cite{liu}. 
	\section*{Disclosure statement}There is no potential conflict to declare.
	\section*{Funding}This  was supported by a grant of the Romanian Ministry of Research
	and Innovation, CNCS--UEFISCDI, project number
	PN-III-P1-1.1-TE-2019-0348, within PNCDI III.
	\section*{Appendix}
	
	\noindent\textbf{Backstepping vs. direct-proportional control design. }  The feedback law designed in \cite{book} is given a priori in the form
	$$u(t,x)=\left<AY(t), \Phi(x)\right>_N,\ t>0,$$for $x$ on the whole boundary or a part of it only. $u$ is plugged into the equations, and  is shown that it assures the stability of the system. Because of its special form it is called "proportional".  It can be equivalently rewritten as 
	$$u(t,x)=\sum_{k=1}^N\int_{\mathcal{O}}y(t,\xi)\varphi_k(\xi)d\xi \Phi_k(x),$$ where, of course, the boundary functions $\Phi_k$ are known. Or, furthermore, 
	$$u(t,x)=\int_\mathcal{O}k(x,\xi)y(t,\xi)d\xi,$$where,  $k(x,\xi)=\sum_{k=1}^N\Phi_k(x)\varphi(\xi)$. The kernel, $k$, which defines the controller,  is given a priori (directly) hence the terminology "direct-proportional".  On the other hand, the backstepping technique involves as-well a kernel. Roughly speaking, the idea is to make the transformation $w(t,x)=\int_\mathcal{O}\tilde{k}(x,\xi)y(t,\xi)d\xi$ which leads to a stable equation in terms of $w$. The kernel $\tilde{k}$ is deduced by imposing that $w$ satisfies the targeted stable equation. After finding $\tilde{k}$, one may express the backstepping control as
	$$u(t,x)=\int_\mathcal{O}\tilde{k}(x,\xi)y(t,\xi)d\xi.$$So, the exact form of the controller is given post priori, i.e., in an indirect way since one has to solve first a hyperbolic equation in order to deduce $\tilde{k}$.
	
	In conclusion, both stabilizing feedback forms involve some kernels, such that the Volterra transformation  maps the original plant to a target  stable system. In the backstepping case, the kernel is deduced by solving a PDE of hyperbolic type in the Goursat form; while, in the direct-proportional case the kernel is given   a priori,  based on the Riesz basis system.

\end{document}